\documentclass{amsart}
\usepackage{amssymb,amstext,amsmath,amscd,amsthm,amsfonts,enumerate,graphicx,latexsym,stmaryrd}
\usepackage[usenames]{color}
\usepackage[all]{xy}
\newtheorem{thm}{Theorem}[section]
\newtheorem{lem}[thm]{Lemma}
\newtheorem{prop}[thm]{Proposition}
\newtheorem{cor}[thm]{Corollary}
\theoremstyle{definition}
\newtheorem{dfn}[thm]{Definition}
\newtheorem{ques}[thm]{Question}

\newtheorem{rem}[thm]{Remark}
\newtheorem{conv}[thm]{Convention}

\theoremstyle{remark}
\newtheorem*{ac}{Acknowlegments}
\numberwithin{equation}{thm}
\def\add{\operatorname{\mathsf{add}}}
\def\ass{\operatorname{Ass}}

\def\depth{\operatorname{depth}}

\def\G{\mathcal{G}}

\def\height{\operatorname{ht}}

\def\le{\leqslant}
\def\m{\mathfrak{m}}
\def\Min{\operatorname{Min}}

\def\mod{\operatorname{\mathsf{mod}}}

\def\p{\mathfrak{p}}
\def\q{\mathfrak{q}}
\def\u{\mathfrak{u}}

\def\m{\mathfrak{m}}
\def\a{\mathfrak{a}}
\def\b{\mathfrak{b}}

\def\spec{\operatorname{Spec}}
\def\syz{\Omega}

\def\X{\mathcal{X}}

\def\R{\mathcal{R}}

\def\U{\mathrm{U}}
\def\V{\mathrm{V}}
\def\rfd{\operatorname{rfd}}
\def\Rfd{\operatorname{Rfd}}

\def\cmd{\operatorname{cmd}}
\def\grade{\operatorname{grade}}
\def\dim{\operatorname{dim}}

\def\cd{\operatorname{CM-dim}}

\def\L{\mathsf{L}}
\def\s{\mathsf{S}}
\def\G{\mathsf{G}}
\def\Y{\mathsf{Y}}
\begin{document}
\allowdisplaybreaks
\title[Annihilators of local cohomology modules and restricted flat dimensions]{Annihilators of local cohomology modules\\ and restricted flat dimensions}
\author{Glenn ando}
\address{Graduate School of Mathematics, Nagoya University, Furocho, Chikusaku, Nagoya 464-8602, Japan}
\email{m21003v@math.nagoya-u.ac.jp}
%\dedicatory{}
\begin{abstract}
Yoshizawa investigated when local cohomology modules have an annihilator that does not depend on the choice of the defining ideal. In this paper we refine his results and investigate the relationship between annihilators of local cohomology modules and restricted flat dimensions.
\end{abstract}
\maketitle
%\tableofcontents
%%%%%%%%%%%%%%%%%%%%%%%%%%%%%%%%%%%%%%%%%%%
\section{Introduction}
Throughout this paper, let $R$ be a commutative noetherian ring.

The annihilators of local cohomology modules have been widely studied.
For example, Faltings' annihilator theorem \cite{Fal} states that if $R$ is a homomorphic image of a regular ring and $M$ is a finitely generated $R$-module, then, for ideals $\a$ and $\b$ of $R$, there exists an integer $n$ such that $\b^n H_\a^i(M) = 0$ for all integers $ i < \lambda_\a^\b(M)$, where $\lambda_\a^\b(M) = \inf\{\depth M_\p + \height(\a + \p)/\p \mid \p \in \spec R \setminus V(\b) \}$; see also \cite[Theorem 9.5.1]{BS}.

Yoshizawa \cite{TY} investigated the following question by using Takahashi's classification theorem of the dominant resolving subcategories of the category $\mod{R}$ of finitely generated $R$-modules \cite[Theorem 5.4]{RT}.

\begin{ques}
Let $M$ be a finitely generated $R$-module.
Let $\p$ be a prime ideal of $R$.
When does there exist an element $s \in R \setminus \p $ such that  $s H_I^i(M) = 0$ for all ideals $I$ of $R$ and all integers $i < \grade(I,R)$?
\end{ques}

Note that the annihilator $s$ does not depend on the choice of the ideals $I$.

The purpose of this paper is to investigate the relationship between such annihilators of local cohomology modules and large/small restricted flat dimensions.

First, we deal with \cite[Proposition 2.5(4),(5)]{TY}, which states that if $R$ is a ring of finite Krull dimension or a Cohen-Macaulay ring, then the resolving subcategory $\R(\p)$ of $\mod{R}$ is dominant for all prime ideals $\p$ of $R$, where 
    $$
    \R(\p) = \left\{
    M \in \mod R \; \middle|
    \begin{tabular}{c}
      There exists  $s \in R \setminus \p$ such that  $s H_I^i(M) = 0$ \\
      for all ideals $I$ of $R$ and all integers $i < \grade(I,R)$
    \end{tabular}
    \right\}.
    $$
We prove the dominance of $\R(\p)$ without assuming that $R$ has finite Krull dimension or is a Cohen-Macaulay ring.

\begin{thm}\label{20}
Let $\p$ be a prime ideal of $R$.
The resolving subcategory $\R(\p)$ of $\mod R$ is dominant.
\end{thm}

Using the above theorem and large/small restricted flat dimensions, we obtain the following theorem, the second assertion of which removes from \cite[Theorem 4.1]{TY} the assumption that the resolving subcategory $\R(\p)$ of $\mod{R}$ is dominant.
\begin{thm}\label{21}
Let $M$ be a finitely generated $R$-module.
Let $\p$ be a prime ideal of $R$.
Consider the following three conditions.
\begin{enumerate}[\quad\rm(a)]
    \item 
    The module $M$ belongs to $\R(\p)$, that is, there exists an element $s \in R \setminus \p$ such that $s H_I^i(M) = 0$ for all ideals $I$ of $R$ and all integers $i < \grade(I,R)$.
    \item
    The large restricted flat dimension $\Rfd_{R_\p}  M_\p$ is at most $0$.
    \item
    The small restricted flat dimension $\rfd_{R_\p} M_\p$ is at most $0$.
\end{enumerate}
Then the following statements hold.
\begin{enumerate}[\rm(1)]
    \item 
    The implications ${\rm(a)}\Leftarrow{\rm(b)}\Rightarrow{\rm(c)}$ always hold.
    \item
    The implication ${\rm(a)}\Rightarrow{\rm(b)}$ holds if $\grade(\q,R)=\depth{R_\q}$ for all $\q \in \U(\p)$.
    \item
    The implication ${\rm(a)}\Leftarrow{\rm(c)}$ holds if $\cmd{R_\p} \le 1$ or $\cd_{R_\p}M_\p < \infty$.
    \item
    The implication ${\rm(a)}\Rightarrow{\rm(c)}$ holds if $\cd_{R_\p} M_\p < \infty$ and $\grade(\p,R) = \depth R_\p$.
\end{enumerate}
\end{thm}
Here $\cmd$ and $\cd$ stand for Cohen-Macaulay defect and Cohen-Macaulay dimension, respectively. 
The first assertion of Theorem \ref{21} leads to the following corollary, which removes from \cite[Corollary 4.3]{TY} all the assumptions.

\begin{cor}\label{22}
Let $M$ be a finitely generated $R$-module, and let $\p$ be a prime ideal of $R$.
Then there exists an element $s \in R \setminus \p$ such that
$s H_I^i(M) = 0$ for all ideals $I$ of $R$ and all integers $i < \grade(I,R) - \Rfd_{R_\p} M_\p$.
\end{cor}

The organization of this paper is as follows.
In Section 2, we state our convention, basic notions and their properties for later use.
In Section 3, we prove Theorem \ref{20}.
In Section 4, we investigate the relationship between annihilators of local cohomology modules and large/small restricted flat dimensions, and prove a result part of which is Theorem \ref{21}, and Corollary \ref{22}.
%%%%%%%%%%%%%%%%%%%%%%%%%%%%%%%%%%%%%%%%%%%%%%
\section{Basic definitions and properties}\label{p}
In this section, we give several definitions and their properties.
We begin with our convention.

\begin{conv}
All rings are commutative noetherian rings with identity. Let $R$ be a (commutative noetherian) ring. We denote by $\mod R$ the category of (finitely generated) $R$-modules.
All subcategories of $\mod R$ are full and closed under isomorphism.
The symbol $\mathbb{N}$ denotes the set of non-negative integers. 
\end{conv}

First, we recall the notions of a dominant subcategory and a resolving subcategory of $\mod R$.
We denote by $\syz_R^n M$ the $n$-th syzygy module of a finitely generated $R$-module $M$.
Note that $\syz_R^n M$ is uniquely determined up to projective summands.

\begin{dfn}
Let $\X$ be a subcategory of $\mod R$.
\begin{enumerate}[(1)]
\item
We denote by $\add\X$ the subcategory of $\mod R$ consisting of the direct summands of finite direct sums of modules in $\X$.
\item
We say that $\X$ is {\em dominant} if, for each $\p \in \spec{R}$, there exists a non-negative integer $n$ such that $\syz_{R_\p}^n\kappa(\p) \in \add\X_\p$, where $\kappa(\p) = R_\p/{\p R_\p}$ and $\X_\p = \{X_\p \in \mod R_\p \mid X \in \X \}$.
\item
We say that $\X$ is {\em resolving} if it satisfies the following four conditions.
\begin{enumerate}[(i)]
    \item $\X$ contains the projective $R$-modules.
    \item $\X$ is closed under direct summands, that is, if $M$ is an $R$-module belonging to $\X$ and $N$ is a direct summand of $M$, then $N$ also belongs to $\X$.
    \item $\X$ is closed under extensions, that is, for an exact sequence $0 \to L \to M \to N \to 0$ of $R$-modules, if $L$ and $N$ belong to $\X$, then $M$ also belongs to $\X$.
    \item $\X$ is closed under syzygies, that is, if $M$ is an $R$-module belonging to $\X$, then $\syz_R^1 M$ also belongs to $\X$.
\end{enumerate}
\end{enumerate}
\end{dfn}

Next, we recall the subcategory $\R(\p)$ of $\mod{R}$ and the subset $\U(\p)$ of $\spec{R}$ introduced by Yoshizawa \cite{TY} and a property of $\R(\p)$.
\begin{dfn}
Let $\p$ be a prime ideal of $R$.
\begin{enumerate}[(1)]
    \item 
    We define the subcategory $\R(\p)$ of $\mod R$ relative to $\p$ by
    $$
    \R(\p) = \left\{
    M \in \mod R \; \middle|
    \begin{tabular}{c}
      There exists  $s \in R \setminus \p$ such that  $s H_I^i(M) = 0$ \\
      for all ideals $I$ of $R$ and all integers $i < \grade(I,R)$
    \end{tabular}
    \right\}.
    $$
    \item
    We denote by $\U(\p)$ the generalization closed subset $\{ \q \in \spec{R} \mid \q \subseteq \p \}$ of $\spec{R}$.
\end{enumerate}
\end{dfn}

\begin{prop} \rm{(\cite[Proposition 2.5(3)]{TY})\label{30}}
Let $\p$ be a prime ideal of $R$.
Then the subcategory $\R(\p)$ is a resolving subcategory of $\mod R$. 
\end{prop}

\begin{rem}
Let $M$ be a finitely generated $R$-module, and let $\p$ be a prime ideal of $R$.
Since $1_R \in R \setminus \p$, the equalities $H_I^i(M) = 0$ for all ideals $I$ of $R$ and all integers $i < \grade(I,R)$ imply that $M$ belongs to $\R(\p)$.
\end{rem}

Finally, we recall the definition of large and small restricted flat dimensions and some properties.
\begin{dfn}
Let $M$ be a finitely generated $R$-module.
\begin{enumerate}[(1)]
    \item 
    The {\em large restricted flat dimension} of $M$ is defined by 
    $$
    \Rfd_R M = \sup_{\p \in \spec{R}} \left\{\depth{R_\p} - \depth{M_\p} \right\}.
    $$
    Note that $\Rfd_{R_\p}{M_\p} = \sup_{\q \in \U(\p)} \{ \depth{R_\q} - \depth{M_\q}\}$ holds for each $\p \in \spec{R}$.
    \item
    The {\em small restricted flat dimension} of $M$ is defined by 
    $$
    \rfd_R M = \sup_{\p \in \spec{R}} \left\{\grade(\p,R) - \grade(\p,M) \right\}.
    $$
     Note that $\rfd_{R_\p}{M_\p} = \sup_{\q \in \U(\p)} \{\grade(\q R_\p,R_\p) - \grade(\q R_\p,M_\p)\}$ holds for each $\p \in \spec{R}$.
\end{enumerate}
We should remark that the original definitions of large/small restricted flat dimensions are different; they are similar to the definition of flat dimension from which those names come, and the equalities in the above definition turn out to hold.
One has $\Rfd_R M$, $\rfd_R M \in \mathbb{N} \cup \{-\infty \}$, and $\Rfd_R M$, $\rfd_R M = -\infty$ if and only if $M = 0$.
Also, $\Rfd_R M \geq \rfd_R M$.
For the details of large/small restricted flat dimensions, we refer the reader to \cite[Theorem 1.1]{AIL} and \cite[Theorem 2.4, 2.11 and Observation 2.10]{CFF}.
\end{dfn}

%%%%%%%%%%%%%%%%%%%%%%%%%%%%%%%%%%%%%%%%%%
\section{Dominance of $\R(\p)$}\label{g}
Throughout this section, let $M$ be a finitely generated $R$-module, and let $\p$ be a prime ideal of $R$.
In this section, we will prove that the resolving subcategory $\R(\p)$ of $\mod{R}$ is always dominant.
We begin with the following lemma, which will be used several times later.
\begin{lem}\label{1}
One has $\rfd_R M = \sup \left\{ \grade(I,R) - \grade(I,M) \mid I \text{ is a proper ideal of } R \right\}$.
\end{lem}

\begin{proof}
By the definition of small restricted flat dimension, there is an inequality 
$$
\rfd_R M \leq \sup \left\{ \grade(I,R) - \grade(I,M) \mid I \text{ is a proper ideal of } R \right\}.
$$
We show that the reverse inequality holds.
Let $I$ be a proper ideal of $R$ and put $n = \grade(I,M)$.
Then there exists an $M$-regular sequence $x = x_1,...,x_n$ in $I$.
By \cite[Proposition 1.2.10(d)]{BH}, we have $\grade(I,M/xM) = 0$.
Hence there exists $\p \in \ass(M/xM)$ with $\p \supseteq I$ such that $\grade(\p,M) = n = \grade(I,M)$ by \cite[Proposition 1.2.1 and 1.2.10(d)]{BH}.
Since $\p \supseteq I$, we have $\grade(\p,R) \geq \grade(I,R)$.
Hence we obtain $\grade(\p,R) - \grade(\p,M) \geq \grade(I,R) - \grade(I,M)$.
Thus the assertion follows.
\end{proof}

Using the above lemma, we can prove the following proposition.
\begin{prop}\label{2}
Suppose that the module $M$ is nonzero.
Then $\syz_R^s M \in \R(\p)$, where $s:=\rfd_RM$. 
\end{prop}

\begin{proof}
Note that $s \in \mathbb{N}$.
Let $I$ be an proper ideal of $R$.
We show that the inequality $\grade(I,\syz_R^s M) \geq \grade(I,R)$ holds.
Let $s = 0$.
Then we obtain $\grade(I,\syz_R^s M) = \grade(I,M) \geq \grade(I,R)$ by Lemma \ref{1}.
Suppose that $s \geq 1$.
Since $\grade(I,M) + s \geq \grade(I,R)$ by Lemma \ref{1}, the grade lemma \cite[Proposition 1.2.9]{BH} yields
\begin{align*}
\grade(I,\syz_R^s M)
&\geq \min \{ \grade(I,R),\grade(I,\syz_R^{s - 1} M) + 1 \} \geq ... \\
&\geq \min \{ \grade(I,R),\grade(I,M) + s \} \\
&= \grade(I,R).
\end{align*}
Since $H_I^i (\syz_R^s M) = 0$ for all integers $ i < \grade(I,\syz_R^s M)$ by \cite[Theorem 6.2.7]{BS} and the local cohomology functor $H_R^i (-)$ is the zero functor for all integers $i$, we obtain $H_I^i (\syz_R^s M) = 0$ for all ideals $I$ of $R$ and all integers $i < \grade(I,R)$.
Therefore, $\syz_R^s M$ belongs to $\R(\p)$.
\end{proof} 

Now we can archive the purpose of this section by showing the following corollary, which is none other than Theorem \ref{20}.
Note that the subcategory $\R(\p)$ is a resolving subcategory of $\mod R$ by Proposition \ref{30}.
\begin{cor}\label{3}
The resolving subcategory $\R(\p)$ of $\mod{R}$ is dominant.
\end{cor}
\begin{proof}
Suppose that the module $M$ is nonzero.
Put $s = \rfd_R M$ and $r = \Rfd_R M$.
In view of \cite[Corollary 4.6]{RT}, it is enough to show that $\syz_R^r M$ belongs to $\R(\p)$.
By Proposition \ref{2}, we obtain $\syz_R^{s} M \in \R(\p)$.
Since $r \geq s $, we have $r-s \geq 0$ and $\syz_R^r M = \syz_R^{r-s}(\syz_R^s M)$ belongs to $\R(\p)$.
Thus the assertion follows.
\end{proof}

%%%%%%%%%%%%%%%%%%%%%%%%%%%%%%%%%%%%%%%%%%
\section{Relationships of $\R(\p)$ with restricted flat dimensions}\label{}
In this section, we investigate the relationship between annihilators of local cohomology modules and large/small restricted flat dimensions.

First, we introduce two invariants defined similarly to large/small restricted flat dimensions.
\begin{dfn}
Let $M$ be a finitely generated $R$-module.
\begin{enumerate}[\rm(1)]
    \item We set $\Rfd'_R{M} := \sup_{\p \in \spec{R}} \{ \depth{R_\p} - \grade(\p, M) \}$.
    Note that $\Rfd'_R{M} \geq \Rfd_R{M}$, and that $\Rfd'_{R_\p}{M_\p} = \sup_{\q \in \U(\p)} \{ \depth{R_\q} - \grade(\q R_\p, M_\p) \}$ for each $\p \in \spec{R}$.
    \item We set $\xi(\p, M) := \sup_{\q \in \U(\p)} \{ \grade(\q, R) - \grade(\q R_\p, M_\p) \}$ for each $\p \in \spec{R}$.
    Note that $\rfd_{R_\p}{M_\p} \geq \xi(\p, M)$ for each $\p \in \spec{R}$.
\end{enumerate}
\end{dfn}

We can also define similar invariants
$$
\rfd'_R{M}=\sup_{\p \in \spec{R}}\{\grade(\p,R)-\depth{M_\p}\}
$$
for a finitely generated $R$-module $M$, and
$$
\xi'(\p,M)=\sup_{\q \in \U(\p)}\{\grade(\q,R)-\depth{M_\q}\}
$$
for a prime ideal $\p$ of $R$ and a finitely generated $R$-module $M$.
These turn out to coincide with $\rfd_R{M}$ and $\xi(\p,M)$, respectively.
\begin{prop}\label{42}
Let $M$ be a finitely generated $R$-module.
\begin{enumerate}[\rm(1)]
    \item One has $\rfd_R{M} = \sup_{\p \in \spec{R}} \{ \grade(\p, R) - \depth{M_\p} \}$.
    \item One has $\xi(\p, M) = \sup_{\q \in \U(\p)} \{ \grade(\q, R) - \depth{M_\q} \}$ for each $\p \in \spec{R}$.
\end{enumerate}
\end{prop}
\begin{proof}
(1) It is enough to show that $\rfd_R{M} \leq \sup_{\p \in \spec{R}} \{ \grade(\p, R) - \depth{M_\p} \}$.
We take $\p \in \spec{R}$ such that $\rfd_R{M} = \grade(\p, R) - \grade(\p, M)$.
By \cite[Proposition 1.2.10(a)]{BH}, we find $\u \in \V(\p)$ such that $\grade(\p, M) = \depth{M_\u}$.
Since $\p \subseteq \u$, we get $\grade(\p, R) \leq \grade(\u, R)$.
Hence we obtain $\rfd_R{M} = \grade(\p, R) - \grade(\p, M) \leq \grade(\u, R) - \depth{M_\u}$.
Thus the assertion follows.

(2) Let $\p \in \spec{R}$.
It is enough to show that $\xi(\p, M) \leq \sup_{\q \in \U(\p)} \{ \grade(\q, R) - \depth{M_\q} \}$.
We take $\q \in \U(\p)$ such that $\xi(\p, M) = \grade(\q, R) - \grade(\q R_\p, M_\p)$.
By \cite[Proposition 1.2.10(a)]{BH}, we find $\u \in \U(\p)$ with $\q \subseteq \u$ such that $\grade(\q R_\p, M_\p) = \depth{(M_\p)_{\u R_\p}}= \depth{M_\u}$.
Since $\q \subseteq \u$, we get $\grade(\q, R) \leq \grade(\u, R)$.
Hence we obtain $\xi(\p,M) = \grade(\q, R) - \grade(\q R_\p, M_\p) \leq \grade(\u , R) - \depth{M_\u}$.
Thus the assertion follows.
\end{proof}

To state the theorem below, we introduce Cohen-Macaulay defect and Cohen-Macaulay dimension.
\begin{dfn}
\begin{enumerate}[(1)]
    \item 
    For a local ring $R$, the {\em Cohen-Macaulay defect} is defined by $\cmd R = \dim R - \depth R$.
    \item
    We denote by $\cd_R M$ the {\em Cohen-Macaulay dimension} of a finitely generated $R$-module $M$. For the definition and their basic properties, we refer the reader to \cite[\S 3]{AG}.
\end{enumerate}
\end{dfn}

The following theorem is one of the main results of this paper, which is a detailed version of Theorem \ref{21}.
\begin{thm}\label{19}
Let $M$ be a finitely generated $R$-module, and let $\p$ be a prime ideal of $R$.
Consider the following five conditions.
\begin{enumerate}[]
\item
$(\L)'$ There is an inequality $\Rfd'_{R_\p}M_\p \leq 0$. 
\item
$(\L)$ There is an inequality $\Rfd_{R_\p}M_\p \leq 0$. 
\item
$(\s)$ There is an inequality $\rfd_{R_\p}M_\p \leq 0$.
\item
$(\G)$ There is an inequality $\xi(\p,M) \leq 0$.
\item
$(\Y)$ The module $M$ belongs to $\R(\p)$.
\end{enumerate}
Then the following statements hold.
\begin{enumerate}[\rm(1)]
\item
\begin{enumerate}[\rm(a)]
\item
Condition $(\s)$ is satisfied if and only if $H_{IR_\p}^i(M_\p) = 0$ for all ideals $I$ of $R$ and all integers $i < \grade(IR_\p, R_\p)$.
\item
Condition $(\G)$ is satisfied if and only if $H_{IR_\p}^i(M_\p) = 0$ for all ideals $I$ of $R$ and all integers $i < \grade(I, R)$.
\end{enumerate}
\item
The following implications always hold true.
$$
\xymatrix@R-2.5pt{
(\L)'\ar@{=>}[r]&(\L)\ar@{=>}[rd]\ar@{=>}[r]&(\s)\ar@{=>}[r]&(\G)\\
&&(\Y)\ar@{=>}[ru]
}
$$
\item
\begin{enumerate}[\rm(a)]
\item
If $\p \in \Min{R}$, then condition $(\L)'$ is satisfied.
\item
If $\p \in \ass{R}$, then condition $(\s)$ is satisfied.
\item
If $\grade(\p, R) = 0$, then condition $(\G)$ is satisfied.
\end{enumerate}
\item
Suppose that $\cd_{R_\p} M_\p < \infty $. Then the equivalence $(\L) \Leftrightarrow (\s)$ holds true.
Hence the implication $(\s) \Rightarrow (\Y)$ holds true.
\item
Suppose that $\cmd{R_\p} \leq 1 $. Then the equivalences $(\L)' \Leftrightarrow (\L) \Leftrightarrow (\s)$ hold true.
Hence the implication $(\s) \Rightarrow (\Y)$ holds true.
\item
Suppose that $\cd_{R_\p} M_\p < \infty$ and $\grade(\p,R) = \depth R_\p$.
Then the implication $(\G) \Rightarrow (\L)$ holds true.
Hence the equivalences $(\L) \Leftrightarrow (\s) \Leftrightarrow (\G) \Leftrightarrow (\Y)$ hold true.
\item
Suppose that $\grade(\q,R) = \grade(\q R_\p,R_\p)$ for all $\q \in \U(\p)$. Then the equivalence $(\s) \Leftrightarrow (\G)$ holds true. 
Hence the implication $(\Y) \Rightarrow (\s)$ holds true.
\item
Suppose that $\grade(\q,R) = \depth R_\q$ for all $\q \in \U(\p)$. Then the equivalences $(\L)' \Leftrightarrow (\L) \Leftrightarrow (\s) \Leftrightarrow (\G) \Leftrightarrow (\Y)$ hold true.
\item
Suppose that $(R, \m)$ is a local ring and $\p = \m$.
Then the equivalences $(\s) \Leftrightarrow (\G) \Leftrightarrow (\Y)$ hold true.
\item
Suppose that $(R,\m)$ is a local ring with $\dim R \leq 2$.
Then the implication $(\s) \Rightarrow (\Y)$ holds true.
\end{enumerate}
\end{thm}

\begin{proof}
The latter assertions of statements (4)--(7) follow from the former ones and (2).

(1a) There are equivalences
\begin{align*}
    (\s)
    &\Leftrightarrow \grade(IR_\p, M_\p) \geq \grade(IR_\p,R_\p) \text{ for all ideals } I \text{ of } R\\
    &\Leftrightarrow H_{IR_\p}^i(M_\p) = 0 \text{ for all ideals } I \text{ of } R \text{ and all integers } i < \grade(IR_\p,R_\p),
\end{align*}
where the first equivalence follows from Lemma \ref{1} and the second follows from \cite[Theorem 6.2.7]{BS}.

(1b) Suppose that $\grade(\q R_\p, M_\p) \geq \grade(\q, R)$ for all $\q \in \U(\p)$.
Let $I$ be a ideal of $R$.
If $IR_\p = R_\p$, then $\grade(IR_\p, M_\p) = \infty \geq \grade(I, R)$.
Suppose that $IR_\p$ is a proper ideal of $R_\p$.
By the proof of Lemma \ref{1}, we obtain $\q \in \U(\p)$ with $\q \supseteq I$ such that $\grade(\q R_\p, M_\p) = \grade(IR_\p, M_\p)$.
Hence $\grade(IR_\p, M_\p) = \grade(\q R_\p, M_\p) \geq \grade(\q, R) \geq \grade(I, R)$.
Thus there are equivalences
\begin{align*}
    (\G)
    &\Leftrightarrow \grade(IR_\p, M_\p) \geq \grade(I, R) \text{ for all ideals } I \text{ of } R\\
    &\Leftrightarrow H_{IR_\p}^i(M_\p) = 0 \text{ for all ideals } I \text{ of } R \text{ and all integers } i < \grade(I, R),
\end{align*}
where the last equivalence follows from \cite[Theorem 6.2.7]{BS}.

(2) Since $\Rfd'_{R_\p}{M_\p} \geq \Rfd_{R_\p}{M_\p} \geq \rfd_{R_\p}{M_\p} \geq \xi(\p,M)$, $(\L)' \Rightarrow (\L) \Rightarrow (\s) \Rightarrow (\G)$ hold true.
Next, we show that $(\L) \Rightarrow (\Y)$ holds true.
Suppose that $\Rfd_{R_\p}M_\p \leq 0$.
Then we have $\depth{R_\q} - \depth{M_\q} \leq 0$ for all $\q \in \U(\p)$.
Since the resolving subcategory $\R(\p)$ of $\mod{R}$ is dominant by Corollary \ref{3}, it follows from \cite[Theorem 1.1]{RT} that $M$ belongs to $\R(\p)$.
Hence $(\L) \Rightarrow (\Y)$ holds true.
Finally, we show that $(\Y) \Rightarrow (\G)$ holds true.
Suppose that there exists an element $s \in R \setminus \p$ such that $s H_I^i(M) = 0 $ for all ideals $I$ of $R$ and all integers $ i < \grade(I,R) $.
Then \cite[Corollary 4.3.3]{BS} yields $(s/1)H_{IR_\p}^i(M_\p)= 0 $ for all ideals $I$ of $R$ and all integers $ i < \grade(I,R) $.
Since $s \in R \setminus \p$, the element $s/1$ is a unit of $R_\p$.
Hence we get $H_{IR_\p}^i(M_\p) = 0$ for all ideals $I$ of $R$ and all integers $ i < \grade(I,R) $.
Hence $(\G)$ follows from (1b).
Therefore, $(\Y) \Rightarrow (\G)$ holds true.

(3a) Suppose that $\p \in \Min{R}$.
Then we have $\U(\p) = \{\p\}$ and $\depth{R_\p} = 0$.
Hence we obtain $\Rfd'_{R_\p}M_\p = \depth R_\p - \grade(\p R_\p, M_\p) \leq \depth R_\p = 0$.
Thus $(\L)'$ holds.

(3b) Suppose that $\p \in \ass{R}$.
Then we have $\depth{R_\p} = 0$.
Hence we obtain $\rfd_{R_\p}M_\p \leq \sup_{\q \in \U(\p)} \{ \grade(\q R_\p, R_\p)\} = \depth R_\p = 0$.
Thus $(\s)$ holds.

(3c) Suppose that $\grade(\p, R) = 0$.
Let $I$ be an ideal of $R$ and $i$ be an integer with $i < \grade(I, R)$.
If $I \subseteq \p$, then we have $\grade(I, R) \leq \grade(\p, R) = 0$.
Hence we get $\grade(I, R) = 0$ and $i < 0$.
Thus we obtain $H_{IR_\p}^i(M_\p) = 0$.
Suppose that $I \not\subseteq \p$.
Then we have $IR_\p = R_\p$.
Hence the local cohomology functor $H_{IR_\p}^i(-)$ is the zero functor.
Thus we obtain $H_{IR_\p}^i(M_\p) = 0$.
It follows from (1b) that $(\G)$ holds.

(4) By \cite[Theorem 2.8]{CFF} and \cite[Theorem 3.8]{AG}, we have $\Rfd_{R_\p}M_\p = \cd_{R_\p}M_\p = \depth{R_\p} - \depth{M_\p}$.
Then there are inequalities and equalities
\begin{align*}
\cd_{R_\p}M_\p
&= \Rfd_{R_\p}M_\p\\
&\geq \rfd_{R_\p}M_\p\\
&\geq \depth R_\p - \depth M_\p\\
&= \cd_{R_\p}M_\p,
\end{align*}
where the second inequality follows from \cite[Observation 2.12]{CFF}.
Hence we obtain the equalities $\Rfd_{R_\p}M_\p = \rfd_{R_\p}M_\p = \cd_{R_\p}M_\p = \depth R_\p - \depth M_\p$.
Thus $(\L) \Leftrightarrow (\s)$ holds true.

(5) By \cite[Lemma 3.1(ii)]{CFF}, we have $\grade(\q R_\p, R_\p) = \depth{R_\q}$ for all $\q \in \U(\p)$.
Then we obtain $\Rfd'_{R_\p}{M_\p} = \Rfd_{R_\p}{M_\p} = \rfd_{R_\p}{M_\p}$.
Thus $(\L)' \Leftrightarrow (\L) \Leftrightarrow (\s)$ hold true.

(6) If $\xi(\p,M) \leq 0$, then $\depth R_\p = \grade(\p,R) \leq \grade(\p R_\p, M_\p) = \depth M_\p$.
Hence we have $\Rfd_{R_\p}M_\p = \cd_{R_\p}M_\p = \depth R_\p - \depth M_\p \leq 0 $, where the first equality follows from \cite[Theorem 2.8]{CFF} and the second follows from \cite[Theorem 3.8]{AG}.
Thus $(\G) \Rightarrow (\L)$ holds true.

(7) We have $\rfd_{R_\p}M_\p = \xi(\p, M)$.
Thus $(\s) \Leftrightarrow (\G)$ holds true.

(8) We have $\Rfd'_{R_\p}M_\p = \Rfd_{R_\p}M_\p = \rfd_{R_\p}M_\p = \xi(\p, M)$.
Thus $(\L)' \Leftrightarrow (\L) \Leftrightarrow (\s) \Leftrightarrow (\G) \Leftrightarrow (\Y)$ hold true.

(9) Note that $R=R_\m=R_\p$.
We see from (7) that $(\s) \Leftrightarrow (\G)$ holds, and there are equivalences
\begin{align*}
    (\Y)
    &\Leftrightarrow H_{I}^i(M) = 0 \text{ for all ideals } I \text{ of } R \text{ and all integers } i < \grade(I, R)\\
    &\Leftrightarrow H_{IR_\m}^i(M_\m) = 0 \text{ for all ideals } I \text{ of } R \text{ and all integers } i < \grade(I, R)\\
    &\Leftrightarrow (\G),
\end{align*}
where the first equivalence holds since any element of $R \setminus \m$ is a unit of $R$ and the third follows from (1b).

(10) If $\height \p \leq 1$, then $\cmd R_\p = \dim R_\p - \depth R_\p \leq \height \p \leq 1$.
Then $(\s) \Rightarrow (\Y)$ holds true by (5).
Let $\height \p = 2$.
Since $R$ is local and $\dim R \leq 2$, we have $\p = \m$.
It follows from (9) that $(\s) \Rightarrow (\Y)$ holds true.
\end{proof}
Here, we give a remark on when the assumptions posed in the statements of the above theorem are satisfied.
\begin{rem}
Let $M$ be a finitely generated $R$-module, and let $\p$ be a prime ideal of $R$.
\begin{enumerate}[(1)]
    \item 
    Suppose that the local ring $R_\p$ is Cohen-Macaulay.
    Then $\cd_{R_\p}M_\p < \infty$ and $\grade(\q, R) = \grade(\q R_\p, R_\p) = \depth{R_\q}$ for all $\q \in \U(\p)$ by \cite[Theorem 3.9]{AG} and \cite[Theorem 2.1.3]{BH}.
    \item
    Suppose that the $R_\p$-module $M_\p$ has finite projective dimension or more generally finite Gorenstein dimension.
    Then $\cd_{R_\p}M_\p < \infty$ by \cite[Theorem 3.7]{AG}.
    \item
    Let $R = k \llbracket X,Y \rrbracket /(X^2, XY)$, where $k$ is a field, and let $\p = (X,Y)R$.
    Then $R_\p = R$ is not a Cohen-Macaulay ring.
    However, since $\cmd{R} = 1$, we have $\grade(\q, R) = \grade(\q R_\p, R_\p) = \depth{R_\q}$ for all $\q \in \U(\p)$ by \cite[Lemma 3.1]{CFF}.
\end{enumerate}
\end{rem}

Next, we investigate which implication does not necessarily hold.
\begin{prop}\label{18}
Let $R$ be a local ring with $\depth R = 0$.
Then the following assertions hold.
\begin{enumerate}[\rm(1)]
\item
One has $\R(\p) = \mod{R}$.
\item
If $\dim{R} \geq 2$, then there exists $\p \in \spec{R} $ such that $\depth R_\p > 0$.
\item
Suppose $\dim{R} \geq 2$, and let $M = R/\p$ for such a prime ideal $\p$ as in {\rm(2)}.
\begin{enumerate}[\rm(a)]
\item
One has $\rfd_{R_\p}M_\p > 0$.
Hence the implication $(\Y) \Rightarrow (\s)$ does not hold for $M$ and $\p$.
Thus none of the implications $(\Y) \Rightarrow (\L)'$, $(\Y) \Rightarrow (\L)$, $(\G) \Rightarrow (\L)'$, $(\G) \Rightarrow (\L)$ and $(\G) \Rightarrow (\s)$ holds for $M$ and $\p$.
\item
One has $\Rfd_R{M} > 0 = \rfd_R{M}$.
Hence the implication $(\s) \Rightarrow (\L)$ does not hold for $M$ and $\m$.
Thus none of the implications $(\s) \Rightarrow (\L)'$, $(\G) \Rightarrow (\L)'$ and $(\G) \Rightarrow (\L)$ holds for $M$ and $\m$.
\item
One has $\Rfd'_R{R} > 0 = \Rfd_R{R}$.
Hence the implication $(\L) \Rightarrow (\L)'$ does not hold for $R$ and $\m$.
\end{enumerate}
\end{enumerate}
\end{prop}

\begin{proof}
(1) Since $\depth R = 0$, we have $\grade(I,R) = 0$ for all proper ideals $I$ of $R$.
Since the local cohomology functor $H_I^i(-)$ is the zero functor for all ideals $I$ of $R$ and all integers $i < 0$, we obtain $H_I^i(M) = 0$ for all finitely generated $R$-module $M$, all ideals $I$ of $R$ and all integers $i < \grade(I,R)$.
Thus $M$ belongs to $\R(\p)$.

(2) By \cite[Lemma 1.4]{CFF}, there exists $\p \in \spec{R} $ such that $\depth R_\p = \dim R - 1 > 0$.

(3a) We have $\rfd_{R_\p}{M_\p} \geq \depth R_\p - \depth M_\p = \depth R_\p > 0$ by \cite[Observation 2.12]{CFF} and (2).
Since $M$ belongs to $\R(\p)$ by (1), $(\Y) \Rightarrow (\s)$ does not hold for these $M$ and $\p$.
Hence the last assertion follows from Theorem \ref{19}(2).

(3b) We have $\Rfd_R{M} \geq \depth{R_\p} - \depth{M_\p} = \depth{R_\p} > 0$ by (2).
Also, $\rfd_R{M} \leq \depth{R} = 0$ by \cite[(2.12.1)]{CFF}.
Hence $(\s) \Rightarrow (\L)$ does not hold for these $M$ and $\m$.
Thus the last assertion follows from Theorem \ref{19}(2).

(3c) We have $\Rfd'_R{R} \geq \depth{R_\p} - \grade(\p,R) = \depth{R_\p} > 0 = \Rfd_R{R}$ by (2).
\end{proof}

\begin{rem}
The table below forming an $6\times6$ matrix describes the relationships among those five conditions which we have discussed so far.
Here, the symbol ``$\bigcirc$"(resp. ``$\times$") in the $(i,j)$ entry means that the implication from the condition placed in the $(i,1)$ entry to the condition placed in the $(1,j)$ entry always holds (resp. does not always hold).
\begin{center}
\begin{tabular}{|l|c|c|c|c|c|}
  \hline
     &  $(\L)'$ & $(\L)$ & $(\s)$ & $(\G)$ & $(\Y)$\\
  \hline
  $(\L)'$ & $\bigcirc$ & $\bigcirc$ & $\bigcirc$ & $\bigcirc$ & $\bigcirc$\\
  \hline
  $(\L)$  & $\times$ & $\bigcirc$ & $\bigcirc$ & $\bigcirc$ & $\bigcirc$\\
  \hline
  $(\s)$  & $\times$ & $\times$ & $\bigcirc$ & $\bigcirc$ & $?$\\  
  \hline
  $(\G)$  & $\times$ & $\times$ & $\times$ & $\bigcirc$ & $\bullet$\\ 
  \hline
  $(\Y)$ & $\times$ & $\times$ & $\times$ & $\bigcirc$ & $\bigcirc$\\
  \hline
\end{tabular}
\end{center}
Our main interest is whether the implication corresponding to the symbol ``$?$" in the $(4,6)$ entry holds or not, which is the implication $(\s) \Rightarrow (\Y)$. Note that if we find a counterexample to this implication, then it is also a counterexample to the implication $(\G) \Rightarrow (\Y)$, which corresponds to the symbol ``$\bullet$" in the $(5,6)$ entry.
\end{rem}

Finally, we prove the following result, which is none other than Corollary \ref{22}, by using Theorem \ref{19}(2).
This result removes from \cite[Corollary 4.3]{TY} all the assumptions, that is, the assumption that the resolving subcategory $\R(\p)$ of $\mod{R}$ is dominant, that $\sup_{\q \in \U(\p)} \left\{\depth R_\q - \depth M_\q \right\} \geq 0$, and that $\depth R_\q = \grade(\q,R)$ for all $\q \in \U(\p)$.
Note that, in the case $\Rfd_{R_\p} M_\p = -\infty$, we get $-(-\infty)$, which is interpreted as $\infty$.
\begin{cor}\label{5}
\rm{(cf. \cite[Corollary 4.3]{TY})}
Let $M$ be a finitely generated $R$-module, and let $\p$ be a prime ideal of $R$.
Then there exists an element $s \in R \setminus \p$ such that 
$s H_I^i(M) = 0$ for all ideals $I$ of $R$ and all integers $i < \grade(I,R) - \Rfd_{R_\p} M_\p$.
\end{cor}

\begin{proof}
Put $ r = \Rfd_{R_\p} M_\p $.
Note that $r \in \mathbb{N} \cup \{-\infty \}$.
If $ r = - \infty $, then $ M_\p = 0 $.
Hence there exists an element $ s \in R \setminus \p $ such that $ s M = 0 $.
Since local cohomology functors are $R$-linear by \cite[Properties 1.2.2]{BS}, we obtain $ s H_I^i(M) = 0 $ for all ideals $I$ of $R$ and all integers $ i $.
Let $ r \geq 0 $.
We use induction on $r$.
In the case $ r = 0 $, our assertion follows immediately from Theorem \ref{19}(2).
Suppose that $ r \geq 1 $.
Then the depth lemma \cite[Proposition 1.2.9]{BH} yields
\begin{align*}
\Rfd_{R_\p}(\syz_R^1 M)_\p
&= \sup_{\q \in \U(\p)}\{\depth{R_\q} - \depth{(\syz_R^1 M)_\q}\} \\
&= \sup_{\q \in \U(\p)}\{\depth{R_\q} - \inf\{\depth{M_\q} + 1, \depth{R_\q}\}\} \\
&= \sup\{\sup_{\q \in \U(\p)}\{\depth{R_\q} - \depth{M_\q}\}-1,0\} \\
&= \sup\{\Rfd_{R_\p}M_\p - 1, 0\}\\
&= r-1.
\end{align*}
Using the induction hypothesis, we find an element $ s \in R \setminus \p $ such that $s H_I^i(\syz_R^1 M) = 0 $ for all ideals $I$ of $R$ and all integers $i < \grade(I,R) - r + 1$.
We take an exact sequence $0 \to \syz_R^1 M \to P \to M \to 0$ with a finitely generated projective $R$-module $P$.
This provides an exact sequence $H_I^i(P) \to H_I^i(M) \to H_I^{i + 1}(\syz_R^1 M) \to H_I^{i + 1}(P)$ for all ideals $I$ of $R$ and all integers $i$.
Since $P$ is a direct summand of a free module, \cite[Theorem 6.2.7]{BS} yields $H_I^i(P) = 0$ for all ideals $I$ of $R$ and all integers $i < \grade(I,R)$.
Thus, from the above exact sequence, we get an isomorphism $H_I^i(M) \cong H_I^{i+1}(\syz_R^1 M)$ for all $i < \grade(I,R)- 1$.
Hence $s H_I^i(M) \cong s H_I^{i + 1}(\syz_R^1 M)= 0$ for all ideals $I$ of $R$ and all integers $i < \grade(I,R) - r$, as $r \geq 1$.
The proof of the corollary is completed.
\end{proof}
%%%%%%%%%%%%%%%%%%%%%%%%%%%%%%%%%%%%%%%%%%
\begin{ac}
The author is grateful to Takeshi Yoshizawa for useful comments.
The author also appreciates the support of his advisor Ryo Takahashi.
\end{ac}
%%%%%%%%%%%%%%%%%%%%%%%%%%%%%%%%%%%%%%%%%%%%%%%%%%%%%%%%%%%%

%%%%%%%%%%%%%%%%%%%%%%%%%%%%%%%%%%%%%%%%%%%%%%%%%%%%%%%%%%%%
\end{document}